\documentclass{amsart}
\usepackage{amsmath,amssymb}
\usepackage[all]{xy}

\theoremstyle{plain}

\newtheorem{introtheorem}{Theorem}

\newtheorem{theorem}{Theorem}[section]
\newtheorem{proposition}[theorem]{Proposition}
\newtheorem{lemma}[theorem]{Lemma}
\newtheorem{corollary}[theorem]{Corollary}

\newtheorem{question}[theorem]{Question}
\newtheorem*{proposition*}{Proposition}

\theoremstyle{definition}

\newtheorem{example}[theorem]{Example}

\theoremstyle{remark}

\def\B{{\mathcal B}}

\def\pr{{P_*}}

\def\Z{{\mathbb Z}}
\def\Q{{\mathbb Q}}
\def\C{{\mathbb C}}

\def\P{{\mathcal P}}
\def\LF{{\mathrm{EF}}}

\def\map{\mathrm{map}}
\def\Der{\mathrm{Der}}
\def\DDer{\overline{\mathrm{Der}}}

\def\Hom{\mathrm{Hom}}

\def\cat0{\mathrm{cat}_0}

\def\dim{\mathrm{dim}}

\def\im{\mathrm{im}}
\def\aut{\mathrm{aut}_1}
\def\caut{\widetilde{\mathrm{aut}}_1}

\def\B{B\mathrm{aut}_1}

\begin{document}

\title{The Universal Fibration with Fibre $X$ in Rational Homotopy Theory}

\author{Gregory  Lupton}
\email{G.Lupton@csuohio.edu}
\address{Department of Mathematics,
           Cleveland State University,
           Cleveland OH 44115}

\author{Samuel Bruce Smith}
\email{smith@sju.edu}
\address{Department of Mathematics,
   Saint Joseph's University,
   Philadelphia, PA 19131}


\begin{abstract}  Let $X$ be a simply connected space with finite-dimensional    rational  homotopy groups.  Let $p_\infty \colon UE \to \B(X)$   be the universal  fibration of simply connected spaces with fibre $X$.     We give   a  DG Lie algebra model  for the evaluation map $ \omega \colon \aut(\B(X_\Q)) \to \B(X_\Q)$  expressed    in terms of derivations of the relative Sullivan model of $p_\infty$.    We deduce  formulas for    the rational Gottlieb group and for the evaluation subgroups of the classifying space $\B(X_\Q)$ as a consequence.  We also prove that $\C P^n_\Q$ cannot be    realized as $\B(X_\Q)$   for $n \leq 4$ and  $X$ with finite-dimensional    rational  homotopy groups.               \end{abstract}

\maketitle
\section{Introduction.}
Given  a simply connected CW complex $X$ of finite type,  let  $\aut(X)$  denote  the  space  of self-maps of  $X$  homotopic to the identity map.    The  group-like space $\aut(X)$   has  a classifying space $\B(X)$.  The space   $\B(X)$    appears as    the base space of the  universal example  $p_\infty \colon UE \to \B(X)$ of a fibration  of simply connected CW complexes with fibre of the homotopy type of $X$    \cite{St, Dold, May}.     
 
The   classifying space   $\B(X)$  offers a computational challenge in homotopy theory.   When $X$ is a finite complex, $\B(X)$ is of CW type (albeit, generally infinite) and satisfies the localization identity $\B(X_P) \simeq \B(X)_P$ for any collection of primes by work of May \cite{May, May2}.   In rational homotopy theory, models for  $\B(X_\Q)$   are due to  Sullivan, Schlessinger-Stasheff and Tanr\'{e} \cite{Su, SS, Tanre}.
The study of the classifying space using these models is an area of  continued activity (see, e.g., \cite{LS3, Yam2, Liu}).   

We say a space $X$ is {\em $\pi$-finite} if $X$ is a simply connected CW complex and  $\dim\, \pi_*(X_\Q)  < \infty$.
 A $\pi$-finite space $X$ has a finitely generated Sullivan minimal model $\land(V; d)$. If  $X$ is a $\pi$-finite space  then $\B(X_\Q)$ is one also (Proposition \ref{finite}, below).  Consequently,  we may iterate the classifying space construction  for   $\pi$-finite rational spaces.  Our first result  here describes  the passage from $\B(X_\Q)$ to $\aut(\B(X_\Q))$ in the setting of derivations of Sullivan models.  We describe this result briefly now,  with fuller definitions in Section \ref{sec2}.

 The relative Sullivan model   for the universal fibration $p_\infty \colon UE \to \B(X)$  with fibre $X$ a $\pi$-finite space is an inclusion of DG algebras.  We write this model throughout as:
 $$\land (Z; d_\infty)  \to (\land Z \otimes \land V; D_\infty).$$      
Let  $\Der(\land V; d)$ denote  the DG Lie algebra of derivations of $\land(V; d)$ and write $\Der_{\land Z}(\land Z \otimes \land V; D_\infty)$   for  the  derivations of $\land Z \otimes \land V$ vanishing on $\land Z$.          We  will assume   derivation spaces are connected.   Thus we restrict  $\Der^1(\land V;d)$ to the cycles  $Z_1(\Der(\land V; d))$   and set   $\Der^n(\land V; d ) = 0$ for $n \leq 0$.  We  do the same for $\Der_{\land Z}(\land Z \otimes \land V; D_\infty)$.     Define a DG Lie algebra map $$  \pr \colon \Der_{\land Z}(\land Z \otimes \land V; D_\infty) \to \Der(\land V; d)$$    by restricting  a derivation $\theta$ to $\land V$ and composing with the projection $P \colon \land Z \otimes \land V \to \land V.$  

Sullivan showed  the DG Lie algebra $\Der(\land V; d)$ gives a  model for the classifying space (\cite[Sec.7]{Su}, see Theorem \ref{b}, below).  We  extend Sullivan's result to  the following:    
\begin{introtheorem} \label{main} Let $X$    be a    $\pi$-finite space.  The map  $$  \pr \colon \Der_{\land Z}(\land Z \otimes \land V; D_\infty) \to \Der(\land V; d)$$    is a Quillen  model for   $\tilde{\omega} \colon \caut(\B(X_\Q)) \to \B(X_\Q),$ the universal cover  of the evaluation map. 
\end{introtheorem}

We mention two consequences of Theorem \ref{main}.  First we deduce an interesting feature      of the  derivations of the relative Sullivan model $\land(Z;d_\infty) \to (\land Z \otimes \land V; D_\infty)$  for $X$ a $\pi$-finite space.  
 \begin{corollary}   \label{corZ}  The DG Lie algebra  $ \Der_{\land Z}(\land Z \otimes \land V; D_\infty)$  satisfies:   
 \begin{itemize} \item[(i)] $H_*(\Der_{\land Z}(\land Z \otimes \land V; D_\infty))$ is an abelian  Lie algebra  
 \item[(ii)]  There are vector space isomorphisms for $n \geq 1$
 $$ H_{n}(\Der_{\land Z}(\land Z \otimes \land V; D_\infty)) \cong H_{n+1}(\Der(\land Z; d))$$
 \end{itemize}
 \end{corollary}

We   also deduce a formula for the  $n$th {\em Gottlieb group}  of the classifying space $\B(X)$.   Recall the subgroup $G_n(Y) $ of $\pi_n(Y)$ is the image of the map induced on homotopy groups by the evaluation map:   $G_n(Y) = \im\{\omega_\sharp \colon \pi_n(\aut(Y)) \to \pi_n(Y)\}$.

 \begin{corollary} \label{maincor}
Let $X$ be  a $\pi$-finite space.  Then   $$G_{n+1}(\B(X_\Q))\cong \im\{ H(\pr) \colon H_{n}(\Der_{\land Z}(\land Z \otimes \land V; D_\infty)) \to H_{n}(\Der \land (V; d)) \}$$ 
for $n \geq 1.$  \end{corollary}

 Corollary \ref{maincor}     leads to an  obstruction theory for  Gottlieb elements of the classifying space   (Proposition \ref{obstr}, below).  More generally,   we obtain  a description of the   poset of  evaluation subgroups  $ G_*(\xi; X_\Q)   \subseteq \pi_*(\B(X_\Q))$    parameterized by fibrations $\xi$ with fibre $X_\Q$.           We give some examples and results on this poset    in Section \ref{sec3}, complementing work of Yamaguchi in \cite{Yam2}.  
 
 We also  prove  a non-realization result for the classifying space.   
   \begin{introtheorem} \label{CP3}  There is no simply connected,  $\pi$-finite space $X$ such that $$\C P^n_\Q \simeq \B(X_\Q)$$ for $n =2, 3, 4.$
 \end{introtheorem}
 Theorem \ref{CP3}  extends     \cite[Th.2]{LS2} for the case $n=2$.    The case  $n=3$  was  recently obtained,   independently,    in  \cite{Liu}.

The paper is organized as follows.  In Section \ref{sec2}, we introduce our notation and recall some results on the rational homotopy theory of  the space $\B(X)$ and    of the monoid $\aut(p)$ of  fibrewise self-equivalences of a fibration $p$.  We prove Theorem \ref{main}   using  these results together with an   identity from    \cite{BHMP}  that connects these   spaces.     Section \ref{sec3} contains our results on the evaluation subgroups of the classifying space.    We prove Theorem \ref{CP3} in Section \ref{sec4}.

\section{Derivations  and fibrewise self-equivalences}   \label{sec2}   
May's  localization equivalence  $\B(X)_\Q \simeq \B(X_\Q)$  for $X$ finite  \cite[Th.4.1]{May2} implies one may study the rationalization of the classifying space using algebraic models with this restriction.    We are    interested here  in the   space $\aut(\B(X))$.  We cannot   expect  to have  $\aut(\B(X_\Q)) \simeq \aut(\B(X))_\Q$ even for $X$ finite, since 
  $\B(X)$ is generally of  infinite  CW type.    Thus in what follows, we   state our main results for rationalized spaces $X_\Q$ for which the various constructions can be made algebraically.  
  
 We establish notation for working in rational homotopy theory.  Our overriding reference for this material is \cite{FHT}. 
Let  $X$ be simply connected and and  CW complex of finite type.   A  \emph{Sullivan  model}      for $X$ is  a  DG algebra $\land (V; d)$ freely generated by  the graded space  $V$  with differential     $d$ satisfying   the     nilpotence  condition (\cite[p.138]{FHT}) and  such that there is  a quasi-isomorphism $\land(V; d)  \to A_{PL}(X)$ with the latter the de Rham  algebra of 
 rational differential forms on $X$ \cite[p.122]{FHT}. A Sullivan model for   a map $f \colon X \to Y$ is a map  of Sullivan models making the diagram of quasi-isomorphims with $A_{PL}(f) \colon A_{PL}(Y) \to A_{PL}(X)$ commute  (see \cite[Ch.23]{FHT}). 
A Sullivan  model $\land (V; d)$  for $X$ is the  {\em Sullivan minimal model} if the differential $d$    is decomposable.    The   homotopy type of   $X_\Q$  is completely determined by a   Sullivan minimal  model $\land(V; d)$.

 A fibration  $p \colon E \to B$  of simply connected spaces with fibre $X$   has a {\em relative Sullivan model} which is an inclusion    $\land (W; \hat{d})  \to  (\land W \otimes \land V; D)$ 
 of DG algebras in which   $\land(W; \hat{d})$  is  a   Sullivan minimal model for the base $B$.  The differential satisfies  $D(w) = \hat{d}(w)$ for $w \in W$ while $D(v) - d(v) \in \land^{+}W \cdot (\land W \otimes \land V)$ 
 for $v \in V.$   The differential $D$ is not generally decomposable but the  DG algebra  $(\land W \otimes \land V; D)$ is a Sullivan model for the total space $E$ (\cite[Ch.14]{FHT}).

Quillen's   framework  for rational homotopy theory is  the category of connected DG Lie algebras.   An object here is a pair  $(L, \partial)$ with  $L = \bigoplus_{n \geq 1} L_{n}$ equipped with a homogenous bracket   and differential $\partial$ lowering degree by one  \cite[p.383]{FHT}.   The commutative cochains functor may be applied to a DG Lie algebra $(L, \partial)$ to obtain        a Sullivan algebra $C^*(L,  \partial) = \land(sL; d = d_0 + d_{ [ \, , \, ]})$   \cite[Lem.23.1]{FHT}.  Here $sL$ is the graded vector space suspension of $L$, $d_0$ is dual to $\partial$ while $d_{ [ \, , \, ]}$ is induced by the bracket in $L$.      A	 DG Lie algebra $L, \partial$ is  a {\em Quillen model}  for $X$ if $C^*(L, \partial)$ is a Sullivan model for $X$.  In this case, we have an isomorphism $\pi_*(\Omega X) \otimes \Q \cong H_*(L, \partial).$   The Quillen model for a map $f \colon X \to Y$ is a DG Lie algebra map $\psi \colon L_X \to L_Y$ such that the induced map $C^*(\psi) \colon C^*(L_Y, \partial_Y) \to C^*(L_X, \partial_X)$ gives a commutative diagram with quasi-isomorphisms to the de Rham forms as for Sullivan models.

Beginning with a   Sullivan  minimal model  $\land (V; d)$ we obtain  the  DG Lie algebra  $\Der( \land V; d)$ defined as   follows:   In degree $n$,  $\Der^n(\land V; d)$ consists of linear self-maps $\theta$ of $\land V$  reducing degrees by $n$,   $\theta (\land V)^{m} \subseteq (\land V)^{m-n}$,  and satisfying   the derivation law $\theta(\chi_1\chi_2) = \theta(\chi_1)\chi_2 + (-1)^{n|\chi_1|} \chi_1\theta(\chi_2)$ for $\chi_1, \chi_2 \in \land V.$   The bracket of two derivations is defined by  the rule   $[\theta_1, \theta_2] = \theta_1 \circ \theta_2 - (-1)^{|\theta_1||\theta_2|}\theta_2 \circ \theta_1.$   The  differential $\delta$  is given  by   $\delta(\theta) = [d,\theta]$ for $\theta \in \Der(\land V).$ 
As we will only consider connected DG Lie algebras, we  restrict in degree $1$ to those $\theta$ with $\delta(\theta) = 0.$ To ease notation, we write $\Der(\land V; d) =   \Der(\land V), \delta$  for the  connected DG Lie algebra.  Sullivan's original result on the classifying space is the following: 
 
  \begin{theorem} {\em \cite[Sec.7] {Su}\label{B}} Let $X$ be  simply connected  and  of finite type with Sullivan minimal model $\land (V; d).$  There is       an isomorphism of graded Lie algebras $$\pi_*(\Omega B \aut (X_\Q)) \cong H_*(\Der(\land V; d)). \qed$$     
\end{theorem}
  
Theorem \ref{B} strengthens to the following statement by the work of several authors:
 \begin{theorem}   \label{b} Let $X$ be  simply connected  and of finite type.  Then $\Der(\land V; d)$ is a Quillen model for $\B(X_\Q)$.    
\end{theorem}
\begin{proof}  Schlessinger-Stasheff   and Tanr\'{e} constructed  a  Quillen model for $\B(X_\Q)$  written $\mathrm{cl}(L_X; \partial_X)$ (see \cite[Cor.7.4(4)]{Tanre}).  Gatsinzi \cite[Th.1]{Gat} constructed a quasi-isomorphism from  $\Der(\land V; d)$ to $\mathrm{cl}(L_*(\land V; d))$   where $L_*( \_ \_)$ is the Quillen functor from DG algebras to DG Lie algebras.  
\end{proof}

  Given a graded vector space $V,$   write $\max{(V)} = \max{\{n \mid V^n \neq 0\}}.$    We have    
    \begin{proposition} \label{finite}  Let  $X$ be simply connected and $\pi$-finite. Then $\B(X_\Q)$  is $\pi$-finite and, we may construct its classifying space $\B(\B(X_\Q)).$  Further, if  $N = \max{(\pi_*(X_\Q))},$     then   $$ (i)  \ \  \max{(\pi_*(\B(X_\Q)))} = N-1 \hbox{\ \  and \ \ }   (ii) \ \ \pi_{N-1}(\B(X_\Q)) \cong \pi_N(X_\Q).$$ 
   \end{proposition}
 \begin{proof}
 Parts (i) and (ii)   are direct consequence of Theorem \ref{B} (cf. \cite[Pro.2.2]{LS3}).  We note that if $V \cong \pi_*(X_\Q)$ then  $H_*(\Der(\land V; d))$ is a sub-quotient of $\mathrm{Hom}(V, \land V)$ and so finite-dimensional.  Thus $\pi_*(\B(X_\Q))$ is finite-dimensional by Theorem \ref{b}. Let $\land(Z; d_\infty)$ denote the Sullivan minimal model for  $C^*(\Der(\land V; d)).$  Then $\B(\B(X_\Q))$ is the rational space with Quillen model $\Der(\land Z;d_\infty).$ Finally, note that the spatial realization of a finitely generated Sullivan model is a CW complex  \cite[p.247-8]{FHT}.
 \end{proof}

 Next we consider the monoid of fibrewise equivalences.   Given a fibration $p \colon E \to B$  set   $$\aut(p) = \{ f \colon E \to E \mid p \circ f =f, f \simeq 1_E\} \subseteq \map(E, E). $$   Let $\B(p)$ denote the classifying space for this monoid.    The main result of \cite{BHMP},  specialized to universal covers, is the following identity:
\begin{theorem}{\em \cite[Th.4.1]{BHMP}} \label{BHMP} Let $p \colon E \to B$ be a fibration of simply connected CW complexes    with fibre $X$. There is a weak homotopy equivalence $$\B(p) \simeq_w \widetilde{\map}(B, \B(X); h)$$
where the latter space is the universal cover of the function space component of    the classifying map $h \colon B \to \B(X)$ for the fibration $p$.   $\qed$ 
\end{theorem}

Sullivan's result,       Theorem \ref{B}  above, extends to an identification for the monoid $\aut(p)$  by the main result of \cite{FLS}.  We recall this result now.  Given  $p \colon E \to B$  with relative Sullivan model   $\land (W; \hat{d})  \to  (\land W \otimes \land V; D)$, 
define 
$\Der_{\land W}(\land W \otimes \land V; D)$ to be the sub-DG Lie algebra of $\Der(\land W \otimes \land V; D)$ 
obtained by restricting to derivations $\theta$ with $\theta(W) = 0.$  The differential $\delta$ is the restriction of the differential for $\Der(\land W \otimes \land V; D)$.   We continue  to restrict, in degree $1$, to the kernel of $\delta$.  We have:  \begin{theorem}{\em \cite[Th.4.1]{FLS}}    \label{B(p)}    Let $p \colon E \to B$ be a fibration of simply connected CW complexes with fibre $X$    and $p_\Q \colon E_\Q \to B_\Q$ the rationalization of $p.$  
 There is  a natural isomorphism of graded Lie algebras in positive degrees:  $$   \pi_*(\aut(p_\Q)) \cong H_*(\Der_{\land W}(\land W \otimes \land V; D)). \qed$$
\end{theorem}

The identification  given in  Theorem \ref{B(p)}  is natural with respect to maps induced by pull-backs of fibrations   \cite[Pro.1.5]{Yam2}.    A  map  $f\colon B' \to B$ into the base $B$ of a fibration $p \colon E \to B$ with fibre $X$     induces a multiplicative map $\aut(p) \to \aut(p')$ where $p'$ is the   pull-back. Then     $f$   induces the map of derivation spaces 
$$f_* \colon \Der_{\land W}(\land W \otimes \land V; D) \to \Der_{\land W'}(\land W' \otimes \land V;D')$$  obtained by composing a derivation with    $f^* \otimes 1$ where $f^* \colon \land( W; \hat{d})\to \land (W'; \hat{d}') $ is   a Sullivan model of $f$.   
 In particular, the inclusion of the base-point in $B$ induces the DG algebra map $\pr \colon  \Der_{\land W}(\land W \otimes \land V;D) \to \Der(\land V;d)$
given by $\pr(\theta) = \pr \circ \theta$ with $P \colon \land W \otimes \land V \to \land V$ the projection.  The map $\pr$ is the  subject of Theorem \ref{main} which   we prove now. 

  \begin{proof}[Proof of Theorem \ref{main}]     
   Let $X$  be simply connected,  $\pi$-finite and of finite type.  Let $\land(Z; d_\infty) \to (\land Z \otimes \land V; D_\infty)$ be the relative model for the universal fibration with fibre $X$.   Recall we are to prove the map $$\pr \colon \Der_{\land Z}(\land Z \otimes \land V; D_\infty) \to \Der(\land V; d)$$ is a Quillen model for   $\tilde{\omega} \colon \caut(\B(X_\Q)) \to \B(X_\Q)$.  Since   $\caut(\B(X_\Q))$ is an H-space, a Quillen model for this space is just    a DG Lie algebra with the correct homotopy groups.   Applying     Theorem \ref{BHMP} to the identity map, we obtain  a   weak equivalence:  $$\caut(\B(X_\Q)) = \widetilde{\map}(\B(X_\Q), \B(X_\Q); 1) \simeq_w \B((p_\infty)_\Q).$$     Applying Theorem \ref{B(p)},  we deduce    that  $$\pi_*(\caut(\B(X_\Q))) \cong H_*(\Der_{\land Z}(\land Z \otimes \land V; D_\infty)) ,$$ as needed.   
 
Finally, Theorem   \ref{b} and  the naturality of the identification in Theorem  \ref{B(p)}, mentioned above, gives that $\pr$ is a Quillen model for $\tilde{\omega}$.          \end{proof}
 
\begin{proof}[Proof of Corollary \ref{corZ}]
Since $\Der_{\land Z}(\land Z \otimes \land V; D_\infty)$ is a Quillen model for the H-space $\aut(p_\infty)$ it has  vanishing brackets in homology.    The isomorphism in Corollary \ref{corZ} (ii) follows from   the chain of isomorphisms:
$$\begin{array}{llll} H_{n}(\Der_{\land Z}(\land Z \otimes \land V; D_\infty)) &\cong \pi_{n}(\aut((p_\infty)_\Q))  \\ & \cong \pi_{n}(\Omega \B((p_\infty)_\Q)) \\ 
& \cong \pi_{n+1}(\aut(\B(X_\Q))) \\
  & \cong H_{n+1}(\Der(\land Z; d_\infty)).
 \end{array} $$
\end{proof} 
  Corollary \ref{maincor} follows directly from Theorem \ref{main} and the definition of the Gottlieb group.

 We  next give  a partial description of the differential $D_\infty$ in terms of derivations. For any relative model $\land(W; d) \to (\land W \otimes \land V; D)$,    the minimality condition for $D$ implies for each $v \in V$ we have $  D(v) =d(v) +  \sum_{i=1}^{s} \theta_{w_i}(v)w_i + \Phi(v)$   where $\Phi(v)$  is in the ideal $(\land^+W \cdot\land^+W)$ of $\land W \otimes \land V$.   
 The linear maps   $\theta_{w_i} \colon \land V \to \land V$ are of degree $|w_i|-1$ and extend to degree $|w_i| -1$ cycles  of   $\Der(\land V;d)$. The map  
$$ w_i \mapsto \langle \theta_{w_i} \rangle \colon W^{|w_i|}  \to H_{|w_i|-1}(\Der(\land V; d))$$ corresponds  to the map induced on rational homotopy groups by the classifying map $h \colon B \to \B(X)$ \cite[Th.3.2]{LS2}.  For the universal  fibration with fibre $X$,  the map $Z^*  \to H_{*-1}(\Der(\land V; d))$ is thus an isomorphism.   Writing $H_*(\Der(\land V; d)) = \langle \theta_1, \ldots, \theta_n \rangle$ in a homogeneous basis, we conclude there is a corresponding basis   $Z = \langle z_1, \ldots, z_n \rangle$ with $|z_i| = |\theta_i| + 1$ such that for $v \in V$ 
\begin{equation} \label{eqD} D_\infty(v) = d(v) + \sum_{i=1}^{n}\theta_i(v)z_i + \Phi(v) \hbox{\, for \, } \Phi(v) \in (\land^+W \cdot \land ^+W) \cdot\land W \otimes \land V. \end{equation}
  
We use this description of $D_\infty$   in the following simple example  illustrating Corollary \ref{corZ}.    In what follows, we  write $(v, P)$ for  the derivation obtained by  sending $v \in V$ to $P \in \land V$ (or $P \in \land Z \otimes \land V$) and  vanishing on a complementary subspace of $V$.     We write $v^* = (v, 1).$    Note that $|(v, P)| = |v| - |P|.$  

 \begin{example} Let $X = S^3 \times \C P^2$.  Write the Sullivan minimal model for $X$ as  $\land(x_2, y_3, z_5; d)$ with subscripts indicating degree and differential given by  $d(x) = d(y) = 0, d(z) = x^3.$    We  see  $H_*(\Der \land(V; d)) = \langle (y, x), (z, y), y^*, (z, x), z^* \rangle.$
 with one non-trivial bracket  $z^* = [(z, y), y^*].$     We   then compute $C^*(\Der (\land V; d))$  and  obtain a (minimal) model for $\B(X_\Q)$ of the form  $\land(Z; d_\infty) = \land(a, b, u, v, w; d_\infty)$ with $|a| = 2, |b| =3, |u| = |v| = 4,  |w| = 6$
 and $d_\infty(a)= d_\infty(b) = d_\infty(c) = d_\infty(u) = d_\infty(v) = 0$ and $d_\infty(w) = bu.$  
 Using (\ref{eqD}), we see that universal fibration has relative Sullivan  model 
 $$\land(a, b, u, v, w; d_\infty) \to (\land(a, b, u, v, w) \otimes \land(x, y, z);  D_\infty)$$
 with $D_\infty = d_\infty$ on $\land(a, b, u, v, w)$,  $D_\infty(z) = w + vx + by + x^3,  D_\infty(y) = u + ax$ and $D_\infty(x) = 0.$  
 We can now  confirm Corollary \ref{corZ}: 
 $$
 \begin{tabular}{cc} $H_*(\Der_{\land Z}(\land Z \otimes \land V; D_\infty))$ & $H_*(\Der (\land V; d))$ \\
 \begin{tabular}{|c|c|} \hline degree & derivation classes \\ \hline
 5 & $z^*$ \\  \hline
 4 &  \\ \hline
 3 & $(z, x), (z,a)$ \\ \hline
 2 & \\ \hline 
 1 & $(z, u) -(z, ax), (z, a^2), $ \\ 
  & $(z, d) - (y, a), (z, x^2)$ \\
 \hline
 \end{tabular} & \begin{tabular}{|c|c|} \hline degree & derivation classes \\ \hline
 6 & $w^*$ \\ \hline
 5 &  \\ \hline
 4 & $(w, a), v^* $\\ \hline
 3 & \\ \hline
 2 &   $a^*, (w, a^2), (w, a),$ \\ & $ (w, v)$ \\ \hline
 \end{tabular}
 \\
 \end{tabular}$$
 Note also that $H_*(\Der_{\land Z}(\land Z \otimes \land V; D_\infty))$ is abelian.  
 \end{example}

  Theorem \ref{main} implies   a formula for a Quillen model for  the universal cover of the monoid $\aut^*(\B(X_\Q))$  of basepoint-preserving automorphisms of the classifying space.  
Write 
$$\DDer_{\land Z}(\land Z \otimes \land V; D_\infty) = \{ \theta \in \Der_{\land W}(\land W \otimes \land V) \mid \theta(v) \subset \land^+W \cdot(\land W \otimes \land V)  \}$$
with the induced differential.       \begin{corollary} \label{eval*}     Let $X$ be a $\pi$-finite space.  Then $\DDer_{\land Z} (\land Z \otimes \land V;  D_\infty)$ 
  is a Quillen model for $\caut^{\,*}(\B(X_\Q))$.    \end{corollary}
\begin{proof} 
The isomorphism of graded Lie algebras  $$\pi_*(\aut^*(\B(X_\Q))) \cong H_*(\DDer_{\land Z} (\land Z \otimes \land V; D_\infty))  $$ is a consequence of  Theorem \ref{main} and the 5-lemma applied to the   long exact homotopy sequence of the evaluation fibration $$\aut^*(\B(X_\Q)) \to \aut(\B(X_\Q)) \to \B(X_\Q).$$     \end{proof}
   
 \section{The Evaluation  Subgroups  of the classifying space} \label{sec3}
 The  Gottlieb group        plays a  central role  in the theory of  fibrations, as   $G_*(X)$   corresponds to the     universal image  of   connecting homomorphisms for   fibrations with fibre $X$  \cite[Th.2.]{Got}.    The rational Gottlieb groups     are the subject of a well-known structure theorem in rational homotopy theory.  For $X$ a finite complex,    $G_{\mathrm{even}}(X_\Q) = 0$ and $\mathrm{dim}\, G_{\mathrm{odd}}(X_\Q) \leq \mathrm{cat}(X_\Q)$ \cite[Pro.29.8]{FHT}.    The     significance of the Gottlieb group of  the classifying space   
  is less clear.         We  give  some examples and results here  to suggest   the rational  Gottlieb group  and, more generally,  the rational evaluation subgroups  of the classifying space    offer   interesting invariants of the  homotopy theory of fibrations.  
  
  We begin with a description of $G_*(\B(X_\Q))$  in terms of derivations,  assuming the identification: $G_n(\B(X_\Q)) \subseteq \pi_n(\B(X_\Q)) \cong H_{n-1}(\Der (\land V; d)).$      
  \begin{theorem} \label{obstr}  A cycle   $\theta  \in \Der^{n-1}(\land V;d)$  represents an element of $G_{n}(\B(X_\Q)) $    if and only if  $\theta$ extends  to a  cycle $\hat{\theta}$ in $\Der^{n-1}_{\land W}(\land W \otimes \land V; D)$  for  every  relative model $\land(W; \hat{d}) \to (\land W \otimes \land V;D)$.  
   \end{theorem}
   \begin{proof} A fibration $\xi$ pulled back from the universal gives a factorization of monoids of fibrewise equivalences: $\aut(p_\infty) \to \aut(p) \to \aut(X).$  The result now follows from Theorems \ref{main} and  \ref{B(p)}.    
   \end{proof} 
   
 Theorem \ref{obstr} roughly implies that, the more ample the   fibrations with fibre $X_\Q$,  the fewer Gottlieb elements in $H_*(\Der(\land V;d)).$    When $X$ is an H-space, fibrations with fibre $X$ are abundant and we have:        \begin{theorem} \label{H}  Let $X$ be a simply connected $\pi$-finite space with $X_\Q$ an H-space.  Then 
 $G_n(\B(X_\Q)) = 0
 $ for $n > N -1$ and $G_{N-1}(\B(X_\Q)) \cong  \pi_{N}(X_\Q)$ where $N = \max{(\pi_*(X_\Q))}.$
 \end{theorem}
 \begin{proof}  The  Sullivan minimal model for $X$ has trivial differential.  The differential $\delta$ for $\Der(\land V; 0)$ is trivial as well.  Let  $\theta \in \Der_n(\land V; 0)$ be a  derivation.  Suppose  $\theta(x) \neq 0$ for some $x \in V^{n}$ with $n < N.$  Take $w$ to have degree $N - |x| +1$ and set $D(v) = wx$ with $D$ vanishing on a complementary subspace to $\langle v \rangle$ in $V$.  For (ii),  we choose an element $y \in V$ appearing in $\theta(x)$.  We then let  $|w| = |y|+1$  and set $D(y) = z$ with $D$ vanishing on a complementary subspace to $\langle y \rangle$ in $V$. 
 In both cases, we see that $\theta$ does not extend to a cycle of $\Der_{\land(w)}(\land (w) \otimes \land V; D),$ as needed.
 \end{proof} 
 
 We note that Theorem \ref{H} can be    proved easily  from the various models for $\B(X)$.    We  may extend the argument above    to give  the following:   
  \begin{theorem} \label{H2}  Let $X$ be a  $\pi$-finite rational H-space and   $Y$ any  $\pi$-finite space. 
  Suppose
 $\max{(\pi_*(X_\Q))} < \max{(\pi_*(Y_\Q))}.$  Then $$G_*(\B(X_\Q \times Y_\Q)) \subseteq  G_*(\B(Y_\Q)).$$ 
 \end{theorem}
 \begin{proof}
 Write the Sullivan minimal model for $X$ as $\land(V; 0)$ and $Y$ as $\land(W'; d')$.  Suppose  $\theta \in \Der_n(\land V \otimes \land W')$ is  a cycle derivation satisfying either (i) $\theta(z) \in (\land V)^{+} \cdot (\land V \otimes \land W')$ for some $z \in W'$ 
 or (ii)  $\theta(x) \neq 0$ for some $x  \in V.$  Define a relative model of the form $\land(w; 0) \to (\land (w) \otimes \land V \otimes \land W'; D)$ where the degree of $w$ depends on the case.    For (i) we pick $v \in V$  where $v \in V$ appears in $\theta(z)$.  Extending $v= v_1$ to a basis of $V$ we set $D(v_1) = w$ and $D(v_i) = 0$ for $i > 1$  with $D = d'$ on $W'$.    For (ii), choose $z \in W'$ of maximal degree and set $D(z) = xw + d'(z)$.  In either case,  we see  $\theta$ does not extend to to a cycle of $\Der_{\land(w)}(\land (w) \otimes \land V \otimes\land Z; D).$  
 \end{proof}  

 At the other extreme from H-spaces, in terms of admitting fibrations with a given fibre, are  the {\em $F_0$-spaces} by which we mean  finite complexes  $X$ which are $\pi$-finite       and  satisfy  $H^{\mathrm{odd}}(X; \Q) = 0$.    The Halperin Conjecture  for $F_0$-spaces  asserts that   $\Der(H^*(X; \Q)) = 0$ for all $F_0$-spaces.   The conjecture has been affirmed in many cases  (see \cite[Prob.1, p.516]{FHT}). 
 
    \begin{theorem} \label{F0} Let $X$ be an $F_0$-space satisfying $\Der(H^*(X; \Q)) = 0.$ 
Then $$G_*(\B(X_\Q)) = \pi_*(\B(X_\Q)).$$  
\end{theorem}
\begin{proof}   By \cite[Pro.2.6]{Me}, $\B(X_\Q)$ is an H-space and so the evaluation map    $\omega \colon \aut(\B(X_\Q)) \to \B(X_\Q)$  has a section given by left multiplication.     \end{proof}   
 
 We   turn to the evaluation subgroups of the classifying space.  
 Let $\LF(X)$   denote the set of  fibre-homotopy equivalence classes of fibrations  $\xi$ with fibre $X$.  The set $\LF(X)$ is partially ordered by  the relation induced by pull-backs.  That is, we define    $ \xi   \leq   \xi' $ if $\xi$ is fibre homotopy  equivalent to the pullback of $\xi'.$      
 Fixing a base space $B$, let $\LF(X; B)$ denote the sub-poset consisting of fibrations $\xi$ over $B$ with fibre $X$.    By  the classification theory \cite{St, Dold, May}, the assignment:  $h \mapsto \xi = h^{-1}(p_\infty)$    induces a natural bijection  $[B, \B(X)] \equiv \LF(X; B).$
By naturality,  if   $[B, \B(X)]$ has  the partial order corresponding to factorization of maps, i.e.,   $h \leq h'$ if there exists $f \colon B \to B$ with $h = h' \circ f$, the above identification is then an isomorphism of posets.

  For any space $Y$,  the Gottlieb group $G_*(Y)$      is  the initial object of  a poset under inclusion of  subgroups of $\pi_*(Y)$  called the {\em evaluation subgroups} of $Y$.   Let   $h \colon B \to Y$ be any map and  write   $\omega \colon \map(B, Y; h) \to Y$ for the evaluation map for the component of the function space.  Define 
 $$G_n(Y; B,  h) = \im\{\omega_\sharp \colon \pi_n(\map(B, Y;h)) \to \pi_n(Y) \}\subseteq \pi_n(Y).$$ 
 Given  maps $h \colon B \to Y$ and $h' \colon B' \to Y,$  we see a factorization $h = h' \circ f$  for $f \colon B \to B'$ implies the reverse inclusion $G_*(Y; B', h') \subseteq G_*(Y; B, h)$ of evaluation subgroups.   
 
 When $Y = \B(X)$,   the evaluation subgroups   are parametrized by   equivalence classes of fibrations  $\xi$ with fibre $X$.    Write $$G_n(\xi ; X) = G_n(\B(X); B, h) \subseteq \pi_n(\B(X))$$ 
 where $h \colon B \to \B(X)$ is the classifying map.   The assignment $  \xi   \mapsto G_*(\xi; X)$  from the poset $\LF(X)$ to the evaluation subgroups of $\B(X),$ partially ordered by inclusion, is   order-reversing.   
 
 In \cite{Yam2}, Yamaguchi introduced a related poset     $G^\xi_*(X)$  of the Gottlieb group $G_*(X)$.  Yamaguchi's groups are recovered, with a shift in degrees, as images: 
  $$G^\xi_{n}(X) = \im \{\Gamma \colon G_{n+1}(\xi; X) \to \pi_n(X) \} \subseteq G_{n}(X)$$
 where     $\Gamma$ is the restriction of   $\omega_\sharp \colon \pi_n(\aut(X)) \to \pi_n(X)$  pre-composed with the isomorphism $\pi_{n+1}(\B(X)) \cong \pi_n(\aut(X)).$   
 We have the identifications: 
   \begin{theorem}  Let $X$ be $\pi$-finite and  $\xi$ be a fibration of simply connected spaces with fibre $X$   with  relative Sullivan model $(\land W; \hat{d}) \to (\land W \otimes \land V; D)$.    Then   
$$\begin{array}{l}  
 G_{n+1}(\xi_\Q; X_\Q) \cong \im \{ H(\pr) \colon H_{n}(\Der_{\land W}(\land W \otimes \land V;D)) \to H_{n}(\Der(\land V; d))\} \\ \\ 
 G^{\xi_\Q}_n(X_\Q) \cong \im \{ \varepsilon^* \circ H(\pr)  \colon H_{n}(\Der_{\land W}(\land W \otimes \land V;D)) \to \Hom(V^{n}; \Q)\}
 \end{array} $$
 with $\pr$  induced by composition   with the projection $P \colon \land W \otimes \land V \to \land V$
 and $\varepsilon^*$  by composition with an augmentation $\varepsilon \colon \land V \to \Q$.  
 \end{theorem} 
 \begin{proof} The first result   follows from  Theorem \ref{B(p)} and the naturality of this identification. The second result is  \cite[Th.1.4]{Yam2}.  
  \end{proof}   
 
  We give  some  examples and results  concerning   the poset $G_*(\xi; X_\Q)$.        Given a set $A$, write $\P(A) = \P(A), \subseteq$ for the power set of partially ordered by inclusion.  We will make use of the order-preserving bijection 
 $\P(\{1, \ldots, n\}) \equiv \Z_2^n,$ where the latter set has   the cartesian product partial order.

   \begin{example} \label{exH1}  Let $X  = S^3 \times S^5 \times S^7.$  We show that  the poset $G_*(\xi; X_\Q)$  is isomorphic to the power set $\P(1, 2, 3, 4)$.   Write the Sullivan minimal model for $X$ as    $\land (V; d) = \land(x_3, y_5, z_7; 0)$  with subscripts denoting degrees.  
 Then $$H_*(\Der(\land V; 0)) = \Der_*(\land V) = \langle z^*, y^*, (z, x), x^*,  (z, y), (y, x)\rangle.$$    
As our base space, we take $B = \B(X_\Q)$ which has Sullivan minimal model  $\land(W; d) = \land(w_1, w_2, w_3, w_4, w_5, w_6; d)$ with  $|w_1|= 4, |w_2|=6, |w_3| = 5, |w_4| = 3, |w_5| = 8, |w_6| = 8$ with  $d(w_i) = 0 \hbox{\, for \,} i = 1, \ldots, 4, \, d(w_5) = -w_3w_1$  and $d(w_6) = -w_4w_2.$  We obtain a family of  relative Sullivan models:     
  $$\xi_{(q_1, q_2, q_3, q_4)}\colon \land (W; d) \to (\land W \otimes \land V; D)$$ by setting
  $ D(x) = q_1w_1, D(y) = q_2w_2,   D(z) = q_3w_3x + q_4 w_4y + q_1q_3w_5 + q_2q_4w_6$
 for   $q_i = 0$ or $1.$  The  order-reversing map $(q_1, q_2, q_3, q_4)  \mapsto G_*(\xi_{(q_1, q_2, q_3, q_4)}; X_\Q)$   then gives a bijection from  $\Z_2^4$ to the set of  distinct evaluation subgroups $G_*(\xi; X_\Q)$.    For   in any relative model  $\land(W; \hat{d}) \to (\land W \otimes V; D),$     if   $D(x) \neq 0$  then   $(z, x)$ and $(y, x)$ are both non-cycles.  On the other hand,   $D(x) = 0$ implies    $y^*$ and $(y, x)$  are either both  non-cycles or both are  cycles depending on the occurrence or non-occurrence  of a non-zero term $wy$ in $D(z).$   
     \end{example}
Following Yamaguchi \cite[Def.1.12]{Yam2},      define the  {\em depth} of the poset  $G_*(\xi; X)$    over a base space $B,$  written $\mathrm{depth}_B \, G_*(\xi; X),$  to be  the number  $n$ in the longest proper chain of subgroups $$G_*(\xi_0; X) \subsetneq \cdots \subsetneq G_*(\xi_n; X)$$ with each $\xi_i$ a fibration over $B$ with fibre $X$.    Example \ref{exH1} gives  $\mathrm{depth} = 4$ for   $G_*(\xi; X_\Q)$ over $\B(X_\Q)$.     Here is one maximal chain:  
    $$\begin{array}{|c||c|c|c|c|c|} \hline
  (q_1, q_2, q_3, q_4) & (1, 1, 1, 1) & (0, 1, 1, 1) & (0,0,1,1) & (0, 0, 0, 1) & (0, 0, 0, 0) \\ \hline
 & z^* & z^*, (z, x)  & z^*, (z, x), & z^*, (z, x), & z^*, (z, x), \\
 G_*(\xi_{(q_1, q_2, q_3, q_4)};X_\Q)           &           &                             & (z, y) &       (z, y), y^* &  (z, y), x^*     \\ & & & & &  y^*, (y, x) \\ \hline
 \end{array}
  $$  
  
For     any  finite  H-space $X$ and any space $B$ by \cite[Ex.5.2]{Yam2} we have:       $$ \mathrm{depth}_B \,  G_*^\xi(X_\Q) = \mathrm{dim}(\pi_*(X_\Q)) - \mathrm{dim}(\pi_N(X_\Q))$$ where $N = \max{(\pi_*(X_\Q))}.$  
 Example \ref{exH1} thus implies a strict inequality:   $$2 = \mathrm{depth}_B(G^\xi_*(X_\Q)) < \mathrm{depth}_B \, (G_*(\xi; X_\Q)) = 4$$
 with $B = \B(X_\Q)$.    
  In fact, we can deduce that 
  \begin{proposition} Given any $M > 0$ there exists a $\pi$-finite H-space  $X$  and a base space $B$ such that
  $$\mathrm{depth}_B (G_*(\xi; X_\Q))\,   >  \mathrm{depth}_B(G^\xi_*(X_\Q))   + M    .$$
  \end{proposition}
  \begin{proof} 
    Given    spaces $X$ and $Y$,   the product fibration  implies the relation: 
$$ \hbox{$\mathrm{depth}_B(G_*(\xi; X \times Y))\geq \mathrm{depth}_B(G_*(\xi; X)) + \mathrm{depth}_B(G_*(\xi; Y))$}$$ (see \cite[Lem.1.13]{Yam2}).   In particular,   
 $$\hbox{ $\mathrm{depth}_B(G_*(\xi; X^m_\Q)) \geq 4m$ while $\mathrm{depth}_B(G^\xi_*(X^m_\Q)) = 2m$}$$
  with $X = S^3 \times S^5 \times S^7$ as in Example \ref{exH1} and $B = \B(X_\Q)$.    
  \end{proof}

Changing the degree of just one generator in Example \ref{exH1} gives  a more complicated example: 
 \begin{example} \label{exH2}  Let $X  = S^3 \times S^5 \times S^9$ with Sullivan model    $\land(x_3, y_5, z_9; 0)$.   Then $$\Der_*(\land V)  = \langle z^*, y^*, x^*, (z, x), (y, x), (z,y),   (z, xy)\rangle.$$ 
 We show the full poset of evaluation subgroups $G_*(\xi; X_\Q)$    is isomorphic to  $P_9 \times \Z_2 $ where $P_9 \subseteq \Z_2^4$ has  Hasse diagram: 
$$P_9 
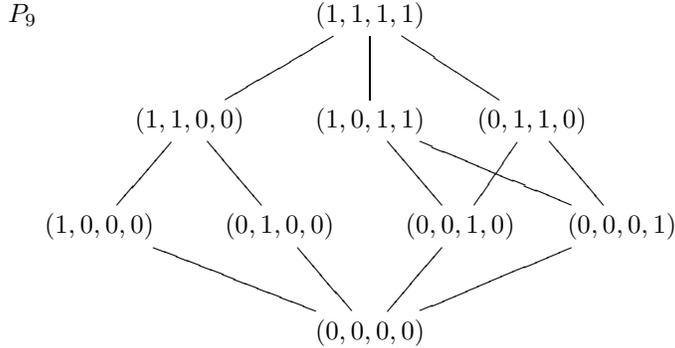
\xymatrix@C=-1.3em{
&&&
(1, 1, 1, 1) \ar@{-}[dll] \ar@{-}[d] \ar@{-}[drrr] &&&
 \\
& (1, 1, 0, 0) \ar@{-}[dr] \ar@{-}[dl]&& 
(1, 0, 1, 1) \ar@{-}[dr] \ar@{-}[drrrr]&&&
(0, 1, 1, 0)\ar@{-}[dll] \ar@{-}[dr]&
\\
(1, 0,0, 0) \ar@{-}[drrr] && (0, 1, 0, 0) \ar@{-}[dr]&&  (0, 0, 1, 0) \ar@{-}[dl]  &&& (0, 0, 0, 1) \ar@{-}[dllll]   \\
&&& (0, 0, 0, 0) &&&}
$$
 
 We explain this briefly.  Let $\land(W; d) \to (\land W \otimes \land V; D)$ be a relative Sullivan model.  Then    $z^* \in G_*(\xi; X_\Q)$ automatically.     
Let $(a_1, a_2, a_3, a_4) \in  \Z_2^4$ record the membership status of the derivations $y^*, x^*, (z, x), (y, x)$ in $G_*(\xi; X_\Q)$ in this order.    We claim the vectors representing realizable  subsets of $G_*(\xi; X_\Q)$   correspond to $P_9$. 
Suppose  $a_1 = 0, a_2 =  1$ so that $y^*$  is not a $\delta$-cycle  and $x^*$  is one.      Then $D(z)$   has  a term involving $y$ alone which implies   $(y, x)$ is a non-cycle  ($a_4 = 0$).  However,  $(z, x)$ is unconstrained as $(z, x)$ is      a cycle exactly when  $D(x) = 0$.  On the other hand,   when  $a_1  =  a_2 =  0$ we can  suppose    $D(z)$ has a term $xy$ and neither $x$ nor $y$ appear elsewhere in the image of $D$.   Such a term does not obstruct $(y, x)$ from being a cycle since $x^2= 0$.  Then, in this case,     $(y, x)$ and $(z, x)$ are both cycles exactly when $D(x) = 0$ and so $a_3$  and $a_4$ are unconstrained.   The allowable vectors with $a_1 = 0$ are thus
$(0, 1, 1, 0), (0, 1, 0, 0), (0, 0, 1, 1), (0, 0, 1, 0), (0, 0, 0, 0)$.  
When $a_1 =1$, the only constraint is that $a_3= a_4$ 
and we obtain the other four vectors in $P_9$, namely $(1, 1, 1, 1), (1, 0, 1, 1), ( 1, 1, 0, 0), (1, 0, 0, 0).$  
 
Now observe   that $(z, y)$ is a cycle precisely when $D(y) =  0$.    The vanishing or non-vanishing of $D(y)$ can be achieved independently of the   terms in $D$ that affect the membership of $y^*, x^*, (z, x), (y, x)$.     Also    $(z, xy)$ is a cycle exactly when both $(z, x)$ and $(z, y)$ are cycles.   
It remains to check that all sets described can be realized as $G_*(\xi; X_\Q)$ for some $\xi$.  This  is straightforward if laborious. 
 \end{example}

  In Example \ref{exH2}, the depth of  $G_*(\xi; X_\Q)$ over the classifying space $\B(X_\Q)$  can be seen to be  $3$  while the depth of the full poset of evaluation subgroups is  $4.$      
 The following result implies   the  maximal depth of   $G_*(\xi; X_\Q)$ over all base spaces  is  the depth     of the full poset $G_*(\xi; X_\Q)$.  
\begin{theorem}  Let $X$ be a $\pi$-finite space.  Then there exists a base space $B$ such that the depth of $G_*(\xi; X_\Q)$ over $B$ equals  the length of the longest chain in the poset $G_*(\xi; X_\Q)$.  
\end{theorem}
\begin{proof}  Let $\xi_0, \xi_1, \ldots, \xi_n$ be fibrations with fibre $X$ giving   a maximal chain of evaluation subgroups.  Writing $p_i \colon E_i \to B_i$ for $\xi_i$ we set $B = B_0 \times \cdots \times B_n$. 
 Let $\xi_i'$ denote the fibration $p_i' \colon E_i \to B$ given by the composition of $p_i$ with the inclusion $B_i \to B.$ Then we see $G_*(\xi_i; X_\Q) = G_*(\xi_i' ; X_\Q).$ \end{proof} 
When $X$ is an  $F_0$-space with $\Der(H^*(X; \Q)) = 0$ the poset   $G_*(\xi; X_\Q)$  is trivial.   For 
 in this case,  $\B(X_\Q)$ is an H-space  and so the evaluation map $$\omega \colon \map(B, \B(X_\Q); h_\Q) \to \B(X_\Q)$$  has a section.  It follows that      $G_*(\xi; X_\Q) = \pi_*(\B(X_\Q))$ for all $\xi.$      
We give an example mixing  even and odd spheres
\begin{example}  \label{exFH} Let $X = S^3 \times S^4 \times S^6 \times S^9$.   We show the poset $G_*(\xi; X_\Q)$ is isomorphic to $\Z_2^4$.   Write the minimal model for $X$ as $\land(x_3, u_4, t_6, v_7, y_9, z_{11}; d)$ with subscripts indicating degrees and $d(u) = d(w) = d(y) = 0, d(v) = u^2, d(z) = t^2.$  In this case: 
$$H_*(\Der(\land V; d))   = \langle z^*, y^*, (z, x), (z, u), v^*,  (y, u), (y, t) (z, xu), x^*, $$ 
$$ \hskip.15in (y, xu), (z, y), (z, ut), (v, t)  \rangle.$$
We also have  $G_*(\B(X_\Q)) = \langle z^*,v^*, (z, u), (v, x) \rangle$.  
Thus $G_*(\xi_\Q; X_\Q)$ contains these  cycles for any  $\xi$. The inclusion or exclusion of  $y^*, (z, y), x^*, (z,x)$ in $G_*(\xi; X_\Q)$ gives the poset $\Z_2^4.$ The status of  $(y, x), (y, u), (y, t), (z, xu), (y, xu),   $  as regards membership in $G_*(\xi; X_\Q),$  depends on the status of these four.  Precisely,  $(y, x)$  and $(y, xu)$ are cycles exactly when both $y^*$ and $(z, x)$ are   cycles,  $(y, u)$ and $(y, t)$  are  cycles exactly when $y^*$ is a cycle, and $(z, xu)$ is a cycle exactly when $(z, x)$ is one.     \end{example}

We conclude this section with an observation regarding the depth of the poset $\LF(X_\Q)$ versus that of the poset of evaluation subgroups:  
  \begin{proposition} Given any $M > 0$ there exists a $\pi$-finite space $X$  such that 
  $$ \mathrm{depth} \,  \LF(X_\Q) \geq  \mathrm{depth} \, G_*(\xi_\Q; X_\Q)  + M .$$
  \end{proposition}
\begin{proof} Let  $m = M+2$ and consider $X = \C P^m$. Then $\B(X_\Q)$ is  an H-space and it is direct to compute that   $\dim(\pi_*(\B(X_\Q))) =  m-1.$   
Write $\pi_*(\B(X_\Q)) = \langle x_1, \ldots, x_{m-1} \rangle$ in a basis.    We may factor the trivial self-map    of  $ \B(X_\Q)$   as a composition   $ h_1 \circ \cdots \circ h_{m-1}$ such that the $m-1$ compositions $H_k = h_1 \circ \cdots \circ h_k$ for $k = 1, \ldots n$ are not homotopic. We do this by   defining $h_k$ to be the map with $(h_k)_\sharp(x_i)= x_i$ for $i = 1, \ldots, k$ and $(h_k)_\sharp(x_j) = 0$ for $j > k.$   
We conclude that $\mathrm{depth}(\LF(X_\Q)) \geq M$ while $\mathrm{depth}(G_*(\xi_\Q; X_\Q)) = 0.$
\end{proof}

\section{A Non-Realization Result for the Classifying Space.} \label{sec4}
An open question in rational homotopy theory  asks:  

\begin{question} {\em \cite[p.519]{FHT}}  Is every simply connected rational homotopy type $Y_\Q$ realized as a classifying space  in the sense that $Y_\Q \simeq \B(X_\Q)$ for some simply connected space $X$? 
 \end{question} 
   In \cite{LS3}, we proved  certain  rational homotopy types, including $\C P^2_\Q,$  could not be  realized if  $X$ is restricted to be a $\pi$-finite space.   Thus to realize these    rational types as a classifying space requires $X$ with infinite-dimensional rational homotopy.   In this section,   we  describe the  relative Sullivan model of the universal fibration under the assumption that  there is a  space $X$ with      $\B(X_\Q) \simeq \C P^n_\Q$ (Proposition \ref{stillu}).  We     apply this description to  prove  Theorem \ref{CP3}, that $\C P^n_\Q$ cannot be realized as $\B(X_\Q)$ for  any $\pi$-finite $X$ and $n = 2, 3$ or $4$. .
   
Write the  minimal Sullivan model for $\C P^n$ as $\land (u_2, v_{2n+1}; \hat{d})$ with $\hat{d}(v) = u^{n+1}$.      Suppose first that  $X$ is a  simply connected space with minimal model $\land (V; d)$. A  fibration $X \to E \to \C P^n$ has   relative Sullivan model   of the form:  $\land(u, v; \hat{d}) \to (\land (u, v) \otimes \land V; D).$   
Let $\chi \in \land V$ and use the minimality condition for $D$ to write:   
$$D(\chi) = d(\chi)  +  u\theta_u(\chi) + u^2\theta_{u^2}(\chi) + \cdots +v \theta_{v}(\chi)  + vu\theta_{vu}(\chi)  + vu^2\theta_{vu^2}(\chi)  + \cdots $$
The maps $\theta_{u^k}, \theta_{vu^k}$ extend to derivations of   $\land V$ of degrees $2k-1$ and $2(n+k) +5,$ respectively.    Taking  $D^2 = 0$ and equating terms with like powers in  the  generators gives a sequence of  relations involving brackets and differentials amongst these derivations.  In particular, we have that $\delta(\theta_u) = \delta(\theta_v) = 0.$  Any  set  of derivations satisfying these identities gives a rational fibration with fibre $X_\Q$.  We have:

\begin{proposition} \label{stillu}  Suppose there exists a simply connected space $X$ with Sullivan model $\land(V; d)$  such that $\B(X_\Q) \simeq \C P^n_\Q$.  Then there exists a relative Sullivan model   $\land(u, v; \hat{d}) \to (\land (u, v) \otimes \land V; D_\infty)$ as above with $\theta_u \in \Der^1(\land V; d)$ and $\theta_v \in \Der^{2n}(\land V; d)$ non-bounding $\delta$-cycles.  Conversely, any relative model  $\land(u, v; \hat{d}) \to (\land (u, v) \otimes \land V; D)$ with  $\theta_u$ non-bounding is a relative Sullivan model for the universal fibration with fibre $X$ and so, in this case, $\theta_v$ is automatically non-bounding. 
\end{proposition}
\begin{proof}  By the description of the differential $D_\infty$ given in (\ref{eqD})  we see that $\theta_u$ and $\theta_v$ represent the non-trivial classes in $\pi_*(\Omega \C P^n_\Q) \cong H_*(\Der(\land V; d)).$ 
Now suppose we are given a relative Sullivan model $\land(u, v; \hat{d}) \to (\land (u, v) \otimes \land V; D)$   with $\theta_u$ not a $\delta$-boundary. The corresponding rational fibration has  classifying map
$h \colon \C P^n_\Q \to  \B(X_\Q) \simeq \C P^n_\Q.$   Since $\theta_u$ is not a $\delta$-boundary,   $h$ induces
an isomorphism on degree two homotopy groups, again  by (\ref{eqD}).   It follows that $h$ is a homotopy equivalence and the given relative Sullivan model is fibre homotopy equivalent to that of the universal.  \end{proof}

For the remainder of the paper, we suppose  $X$ is $\pi$-finite with  $\B(X_\Q) \simeq  \C P^n_\Q$.   Write $\land(V; d)$ for the Sullivan minimal model for $X$.  Then $V^{2n} \cong \Q$ and $V^q = 0$ for $q > 2n$ by Proposition \ref{finite}.   Let  $\land(u, v; \hat{d}) \to (\land (u, v) \otimes \land V; D_\infty)$ denote  the relative Sullivan  model for  the universal fibration with fibre $X$. 
For degree reasons, the only possible non-vanishing derivations are $\theta_u, \theta_{u^2}, \cdots, \theta_{u^n},  \theta_v$  of degrees $1, 3, 5, \ldots, 2n-1$  and $2n,$ respectively, where the  last two are linear maps:   $\theta_{u^n} \colon V^{2n-1} \to \Q$ and $\theta_v \colon V^{2n} \to \Q.$  The identities arising from the equation $D_\infty^2 = 0$ are as follows: 
\begin{equation} \label{identities} \begin{tabular}{lll}
$u$-terms &&    $\delta(\theta_u) = 0$    \\ \\
 $u^k$-terms &&  $\delta(\theta_{u^k}) = \sum_{i+j = k, i \leq j} [\theta_i, \theta_j]$  for  $k = 2, \ldots, n+1$  \\ \\
 $v$-terms  &&  $\delta(\theta_{v}) = 0$ \\
\end{tabular} \end{equation}
 
   Proposition \ref{stillu}  may be refined, in this case,  to the following:        \begin{lemma}  \label{lem1}     In the  relative Sullivan model   $\land(u, v; \hat{d}) \to (\land (u, v) \otimes \land V; D_\infty)$  the derivation $\theta_u \in \Der^1(\land V)$ is not a $\delta$-boundary and $\theta_v  \neq 0$.     Conversely,  any collection $\theta_u, \theta_{u^2}, \cdots, \theta_{u^{n}}, \theta_v$ satisfying the identities (\ref{identities}) with  $\theta_u$ not a $\delta$-boundary  is a relative Sullivan model for the universal fibration with fibre $X$.  Consequently,  $\theta_v \neq 0$.  $\qed$ \end{lemma}

We show that altering   $\theta_u$ by a boundary yields a compatible collection of derivations: 
\begin{lemma} \label{boundary}  Let $\varphi \in \Der^2(\land V)$.  There is a relative Sullivan model for the universal  fibration with fibre $X$ with derivations given by $\theta_u', \theta_{u^2}', \ldots, \theta_{u^n}', \theta_v'$ with  $$\theta_u' = \theta_u + \delta(\varphi).$$ 
\end{lemma} 
  \begin{proof}  Since $\theta_u'$ is a  cycle in $\Der^1(\land V)$ and $H_*(\Der(\land V;d))$ is concentrated in degrees $1$ and $2n$, the derivation cycle $2 [\theta_u', \theta_u'] \in \Der^2(\land V)$   must be a $\delta$-boundary.  Thus we can choose  $\theta_{u^2}' \in \Der^3(\land V)$ with $\delta(\theta_{u^2}') = 2 [\theta_u', \theta_u'].$ Next observe $[\theta_{u}', \theta_{u^2}']$
  is a $\delta$-cycle and so a $\delta$-boundary.  Thus we can find $\theta_{u^3}' \in \Der^5(\land V)$ with $\delta(\theta_{u^3}') = [\theta_{u}', \theta_{u^2}'].$  Continuing in this manner, we obtain a collection $\theta_{u^k}'$ for $k = 1,\ldots, n$ satisfying all but the last identity in (\ref{identities}).     Finally, set $\theta_v' = \sum_{i+j= n, i \leq  j} [\theta_{u^i}', \theta_{u^j}'] $.   By  Proposition \ref{lem1}, these derivations give a relative Sullivan model for the universal  fibration.
  \end{proof}

Regarding  the differential  $d$, we have a  quadratic pairing:  
  \begin{lemma}  \label{pair} Let $y \in V^{2n} \cong \Q$ be nontrivial. Given a basis  $\{ z_1, \ldots, z_n \}$  for $V^{2n-1}$ there is a corresponding basis  $\{x_1, \ldots, x_n \}$ for $V^2$ so   that 
  $$d(y) = x_1z_1 + x_2z_2 + \cdots  + x_nz_n + \hbox{\, terms not involving any $z_j$ }.$$
  \end{lemma}
  \begin{proof}  The derivations $z_j^*$ in $\Der^{2n-1}(\land V)$ cannot be cycles for it is not possible for these
  derivations to be boundaries.  Thus each $z_j$ must appear in $d(y)$  and we have  a pairing as above.  If there is some $x \in V^2$ not in the span of $\{x_1, \ldots, x_n\}$ then $(y, x)$ is a non-bounding cycle of degree $2n-2$, a contradiction.  
  \end{proof}
    The quadratic part of  $d(y)$   also has   terms involving elements of $V^3$ and $V^{2n-2}$:
  \begin{lemma} \label{w} Given $w \in V^3$ there is $\overline{w}  \in V^{2n-2}$ such that $w \overline{w}$ appears in $d(y)$ and $\overline{w}$ does not appear
  in other terms of $d(y)$.
  \end{lemma}
  \begin{proof} Write $d(w) = \sum_{i=1}^{n} q_{i}x_ix'_{i}$ for some $x_i' \in V^2$.      Define $\theta \in \Der^{2n-3}(\land V)$ by the formula
   $$\theta = (y, w) - \sum_{i=1}^{n}q_i (z_i, x_i').$$   We see $\delta(\theta) = 0$ and so  $\theta = \delta(\alpha)$ for some $\alpha \in \Der^{2n-2}(\land V).$   Then   $\delta(\alpha(y)) = -\alpha(d(y)) = w$ implies $\alpha = \overline{w}^* +  \alpha'$  for some $\overline{w} \in V^{2n-2}, \alpha'(V^4) = 0.$    Further we must have the term $w \overline{w}$ with $\overline{w}$
  as specified.     
  \end{proof}
  We apply the preceding to deduce:     
   \begin{lemma}  \label{decomposable} In the relative Sullivan model for the universal fibration with fibre $X$, we may assume that $\theta_u(y)$ decomposable in $\land V$     for  $y \in V^{2n}$ nontrivial.       
   \end{lemma}
  \begin{proof}  Suppose $\theta_u(y) = z + \chi $ for    some $z \in V^{2n-1}$ and $\chi$ decomposable.   Taking  $z = z_1$ and extending to a basis, we  set $\theta_u' = \theta_u -\delta(x_1^*)$  with $x_1 \in V^2$ as in Lemma \ref{pair}.  Then  $\theta_u'(y)$ is decomposable.    Now apply Lemma \ref{boundary} to obtain a compatible collection with $\theta_{u}'$ for the relative model of the universal  fibration.     \end{proof}

Regarding $\theta_{u^2}$, we have:   
\begin{lemma} \label{u2}  If $\theta_u(y)$ is decomposable for $y \in V^{2n}$ nontrivial, then $\theta_{u^{2}}$ vanishes on $V^3.$ 
\end{lemma}
 \begin{proof}
 Suppose $w \in V^3$ satisfies $\theta_{u^2}(w) = 1.$  Then   $D_\infty(w) = u^2 + \theta_u(w)u + d(w)$.   Consider the term $w  \overline{w}$ occurring in $d(y)$ with $\overline{w} \in V^{2n-2}$ from Lemma \ref{u2}.  This term occurs   as a summand of $D_\infty(y).$     
 Applying $D_\infty$ again gives a summand $u^2\overline{w}$ in $D_\infty^2(y)$.   We claim that this term cannot be cancelled.  
For note, for degree reasons,     $u \overline{w}$    can only occur in   $D_\infty(z)$ for $z \in V^{2n-1}$.     Since $\theta_u(y)$ is indecomposable we cannot have a corresponding  term  $uz$  in $D_\infty(y)$.          
 \end{proof}

We apply these results to prove  there is no $\pi$-finite $X$ with $\B(X_\Q) \simeq \C P^n_\Q$ for $n = 2, 3, 4$  
  
 \begin{proof}[Proof of Theorem \ref{CP3}]
 By Lemma \ref{decomposable}, we may assume $\theta_u(V^{2n}) \subseteq \land^+ V \cdot \land^+V$. 
 By Lemma \ref{u2}, this implies $\theta_{u^2}(V^3) = 0.$   The formulas for $\theta_v \colon V^{2n} \to \Q$ given in equation (\ref{identities}) for the cases $n = 2, 3, 4$ are as follows.   
$$ \begin{tabular}{rlc}
$n = 2$ && $ \theta_v = [\theta_u, \theta_{u^2}]$ \\ \\
$n = 3$ && $\theta_v = [\theta_u, \theta_{u^3}] + 2[\theta_{u^2}, \theta_{u^2}]$ \\ \\\
$n = 4$ && $\theta_v = [\theta_u, \theta_{u^4}] + [\theta_{u^2}, \theta_{u^3}]$ 
\end{tabular}
$$ 
Let $y \in V^{2n}$.  Then  $\theta_u(y)$  decomposable  implies  $[\theta_u, \theta_{u^n}](y) = 0$ in each case.
Also,  $\theta_{u^2}(V^3) = 0$ implies       $[\theta_{u^2}, \theta_{u^n}](y) = 0$.  Thus, in all three cases, $\theta_v = 0$, contradicting Lemma \ref{lem1}.     \end{proof}

 \bibliographystyle{amsplain}
\providecommand{\bysame}{\leavevmode\hbox to3em{\hrulefill}\thinspace}
\providecommand{\MR}{\relax\ifhmode\unskip\space\fi MR }
\providecommand{\MRhref}[2]{%
  \href{http://www.ams.org/mathscinet-getitem?mr=#1}{#2}
}
\providecommand{\href}[2]{#2}

\end{document}